\documentclass{article}
\usepackage{amsmath, amssymb, verbatim, textcomp, graphicx, framed}
\usepackage{graphicx}
\usepackage[a4paper,left=2.5cm,right=2.5cm,top=3cm,bottom=2cm]{geometry}
\newtheorem{theorem}{Theorem}
\newtheorem{lemma}{Lemma}
\newtheorem{proposition}{Proposition}

\newtheorem{definition}{Definition}
\newenvironment{proof}{{\sc Proof:}}{~\hfill$\Box$}
\newtheorem{remark}{Remark}
\def\e{\mathrm{e}}
\def\Z{\mathbb{Z}}
\def\N{\mathbb{N}}
\def\ra{\rightarrow}
\def\R{\mathbb{R}}
\begin{document}

\title{Wigner matrices, the moments of roots of Hermite polynomials and the semicircle law
}


\author{\Large{M. Kornyik}  \\  E\"otv\"os Lor\'and University \\ Department of Probability Theory and Statistics \\ P\'azm\'any P\'eter s\'et\'any 1/C., H-1117, Budapest, Hungary \\
              \textit{email}:koma@cs.elte.hu        \and
        \Large{Gy. Michaletzky} \\   E\"otv\"os Lor\'and University \\ Department of Probability Theory and Statistics \\ P\'azm\'any P\'eter s\'et\'any 1/C., H-1117, Budapest, Hungary\\
              \textit{email}:michaletzky@caesar.elte.hu}

\if 0\begin{center}Mikl\'os Kornyik \\
           \\
           Gy\"orgy Michaletzky \\

\end{center} \fi


\maketitle

\begin{abstract}
In the present paper we give two alternate proofs of the well known theorem that
the empirical distribution of the appropriately normalized roots
of the $n^{th}$ monic Hermite polynomial $H_n$ converges weakly to the semicircle law, which is also the weak limit of the empirical
distribution of appropriately normalized eigenvalues of a Wigner matrix. In the first proof -- based on the recursion satisfied by the Hermite polynomials -- we show that the generating
function of the moments of roots of $H_n$ is convergent and it satisfies a fixed point equation, which is also satisfied by $c(z^2)$,
where $c(z)$ is the generating function of the Catalan numbers $C_k$. In the second proof we compute the leading and the second leading term of the
$k^{th}$ moments (as a polynomial in $n$) of $H_n$ and show that the first one coincides with $C_{k/2}$, the $(k/2)^{\rm th}$ Catalan number,
where $k$ is even and the second one is given by $-(2^{2k-1}-\binom{2k-1}{k})$. We also mention the known result that the expectation of the characteristic polynomial ($p_n$) of a
Wigner random matrix is exactly the Hermite polynomial ($H_n$), i.e. $Ep_n(x)=H_n(x)$, which suggest the presence of a
deep connection between the Hermite polynomials and Wigner matrices. 
\end{abstract}

\section{Introduction}

\label{intro}
In random matrix theory to analyse the behaviour of the eigenvalues of a random matrix one possibility is to consider the sum of
the $k^{\rm th}$ powers of its eigenvalues. This can be done either via analysing the trace of the $k^{\rm th}$ power of the
random matrix, or through the $k^{\rm th}$ moments of the roots of its characteristic polynomial.
One can find many results of the former type (see \cite{Bai88}, \cite{mar67}), but the latter has not yet been
thoroughly investigated (\cite{bre00}, \cite{for06}, \cite{hardy15}).
Since there is an implicit connection between the moments and elementary symmetric polynomials of the roots (Newton's identities),
one can make observations of the moments of roots via examining the coefficients of the characteristic polynomial. \\
Let us introduce the following 
\begin{definition} A random symmetric matrix $A=[a_{ij}]_{i, j=1,\ldots,n}$ is called a Wigner matrix, if all the elements $(a_{ij})_{1\leq i \leq j \leq n}$ are independent with zero mean, the elements on the diagonal are identically distributed, and the off-diagonal elements are identically distributed with finite second moments.
\end{definition}


 Forrester and Gamburd proved in \cite{for06} that if $A$ is a Wigner matrix with its off-diagonal elements having variance $c^2>0$ then one has \begin{equation} \label{expcharp} E\det [\lambda I- A]=c^nH_n(x/c), \end{equation}
where $H_n(x)$ is the $n$-th monic Hermite polynomial given by
\begin{equation*}  H_n(x)= \sum_{k=0}^{\lfloor n/2 \rfloor}(-1)^k \binom{n}{2k} (2k-1)!! x^{n-2k}. \end{equation*}
We would like to remark, that in order to get (\ref{expcharp}) it is sufficient to assume independence of all the free elements of $A$, i.e. the independence of $a_{ij}$ for $1\leq i \leq j \leq n$, $E[a_{ij}]=0$ for $1\leq i\leq j \leq n$ and $E[a_{kl}^2]=c^2<\infty$ for $1\leq k <l \leq n$. 
Their proof goes per definition, that is computing
$$ E\det [\lambda I-A]=\sum_{\sigma\in S_n} (-1)^{|\sigma|} E\prod_{i=1}^n{(\lambda_{i}\delta_{i\sigma(i)}-a_{i\sigma(i)})}. $$
Note that although the assumptions on the random matrix are not very restrictive, yet the resulting expectation of the characteristic polynomial is a very specific one, namely the one orthogonal with respect to the density function of the standard normal distribution. This fact suggests the presence of an intrinsic connection between Wigner matrices and Hermite polynomials. 
Hermite polynomials have another interesting property; they coincide with the matching polynomial $M_{K_n}(x) $ of the complete graph $K_n$. 
The matching polynomial of a graph $G=(V,E)$ is defined by $$ M_G(x)=\sum_{k\geq 0} (-1)^k m_k (G)x^{n-2k} ,$$
where $m_k(G)$ denotes the number of matchings with exactly $k$ edges and $n=|V|$.

In the first theorem of Section 1 we are going to present a short, direct proof of the well known theorem stating that the
semicircle law describes the asymptotic distribution of the (normalized) roots of the Hermite polynomials by showing that the generating function (without computing the actual coefficients) of the sum of $k^{th}$ power of the roots
converges to $\sum_{k=0}^\infty{C_kz^{2k}}=c(z^2) $, where $c(z)$ denotes the generating function of the Catalan numbers, using a fixed point argument similar to
Girko's idea in \cite{gir85}.

In the second theorem we explicitly compute the leading and the second leading coefficient in $n$ of the sum of $k^{th}$ power of the roots
and show that the first one is equal to $C_{k/2}$ and the second one is equal to $-(2^{2k-1}-\binom{2k-1}{k})$ by using the implicit connection  between the moments
and elementary symmetric polynomials of the roots of an arbitrary polynomial (also known as Newton's identities or Vi\'eta's formulae).
This result also implies -- after proper scaling -- the weak convergence of the empirical distribution
of the scaled roots of $H_n$ as $n\rightarrow \infty$ to the semicircle law and the convergence rate cannot be faster than $O(1/n)$.

\section{The semi-circle law for the roots of the Hermite polynomials}\label{Hermite_SC}

Let us now consider the roots of the Hermite polynomials. Denote by $\xi^{(n)}_1, \dots , \xi^{(n)}_n$
its zeros and denote by $\mu_n = \frac{1}{n} \sum_{j=1}^n \delta_{\lambda_j^{(n)}}$ the empirical
distribution determined by the normalized roots, where $\lambda_j^{(n)}  =\frac{\xi_j^{(n)}}{2\sqrt{n}}$,
and by $M_n (k) = \sum_{j=1}^n \left(\xi^{(n)}_j\right)^k$ the sum of the $k^{{\rm th}}$
powers.
\begin{theorem}[See \cite{Ga87}]\label{Hermite_limit}
The limit distribution of the empirical distribution of the roots of the Hermite polynomial $H_n$, as $n\rightarrow \infty$, is
given by the semicircle distribution, that is
\begin{align}
        \mu_n\xrightarrow[n\ra \infty]{w} \rho_{sc}(x)dx
   \end{align}
where ' $\xrightarrow[]{w}$ ' means weak convergence and $\rho_{sc}(x)=\frac2\pi \sqrt{1-x^2} \cdot  \mathbf{1}_{[-1,1]}(x)$ with $\textbf{1}_{[-1,1]}(x)$ denoting the indicator function of the set $[-1,1]\subset \R$.
\end{theorem}

\begin{proof}
To prove the weak convergence we apply the methods of moments. First for the sake of the reader we present a direct proof
of the following known lemma (see \cite{la98}) connecting the moments of the empirical distribution of the roots and the corresponding polynomial.
\begin{lemma}\label{powers}
        Let $p(x)=\sum_{j=0}^n{b_jx^j}$ be a monic polynomial (i.e. $b_n=1$) with real coefficients, let $\eta_1, \eta_2, \ldots , \eta_n$
		denote its roots, let $m(k)=\sum_{j=1}^n{\eta_j^k}$, and let $\mathcal{M}(z)=\sum_{k=0}^\infty{m(k)z^k}$
		denote the generating function of the sums of powers of the roots. Then
        \begin{align}
            \mathcal{M}(z)=-\frac{z\widehat{p}'(z)}{\widehat{p}(z)}+n
        \end{align}
        where $\widehat{p}(z)=\sum_{j=0}^n{b_{n-j}z^j}$ denotes the conjugate polynomial.
    \end{lemma}
    \begin{proof}
        According to the Newton identities one has
        \begin{align}\label{Newton}
            \sum_{j=0}^{k}{m(k-j)b_{n-j}}=(n-k)b_{n-k}
        \end{align}
The following computation is straightforward:
        \begin{align}
            \mathcal{M}(z)\widehat{p}(z)&=\sum_{l=0}^\infty{m(l)z^l}\sum_{j=0}^n{b_{n-j}z^j}=\sum_{k=0}^\infty{\sum_{j=0}^{k}m(k-j)b_{n-j}z^k}= \nonumber\\
            &=\sum_{k=0}^\infty{(n-k)b_{n-k}z^k}=n\widehat{p}(z)-z\widehat{p}'(z) \nonumber
        \end{align}
        hence
        \begin{align}
            \mathcal{M}(z)=-\frac{z\widehat{p}'(z)}{\widehat{p}(z)}+n \nonumber
        \end{align}
        so the proof is complete.
    \end{proof}

Let us return to the proof of the proposition. Introduce the notation
$\mathcal{M}_n(z):=\sum_{k=0}^\infty{M_n(k)z^k}$. We are going to show that
\begin{equation}\label{limit_mn}
\frac{1}{n}\mathcal{M}_{n}(z/\sqrt{n})\ra  \sum_{k=0}^\infty C_k z^{2k}\,, \quad \hbox{for} \ 0\leq z \leq \frac{1}{3}\,,
\end{equation}
where $C_k = \frac{1}{k+1} \left(\begin{array}{c} 2k \\ k \end{array}\right)$ is the $k^{th}$ Catalan number.

The proof of this claim will be based on the well-known recursive identities of the (probabilists') Hermite polynomials (similar as in \cite{Mar01}):
\begin{eqnarray*}
H_0(x)& = & 1 \nonumber \\
H_1(x)& = & x \nonumber \\
H_{n+1} (x) &=& x H_n (x) - n H_{n-1} (x) \\
\frac{d}{dx}H_n(x) &=&  n H_{n-1}(x)\,.
\end{eqnarray*}
Denoting by $\widehat{H}_n(x)= x^n H_n (\frac{1}{x})$ the conjugate polynomial
it can be easily checked that
\begin{eqnarray}
\widehat{H}_0(x)& = & 1 \\
\widehat{H}_1(x)& = & 1 \\
\widehat{H}_{n+1} (x) &=& \widehat{H}_n (x) - nx^2 \widehat{H}_{n-1} (x) \label{conj_rec}\\
\frac{d}{dx}\widehat{H}_n (x) &=& \frac{n}{x} \left(\widehat{H}_n(x) - \widehat{H}_{n-1}(x)\right)\label{deriv_rec}\,.
\end{eqnarray}
Since all the roots of $H_n$ are no greater in absolute value than $2\sqrt{n+\frac{1}{2}}$ (See \cite{Sze39} p. 131. Theorem 6.32.) we obtain that the conjugate polynomials
do not vanish in the interval $\left[-\frac{1}{3\sqrt{n}}, \frac{1}{3\sqrt{n}}\right]$.
Since $\widehat{H}_n (0) = 1$, they are in fact positive in that interval. This observation
combined with equation (\ref{conj_rec}) above implies that in this interval
\[
 \frac{\widehat{H}_{n-1}(x)}{\widehat{H}_{n}(x)}\leq \frac{1}{nx^2},
\]
consequently
\begin{align}\label{upperbound}
 \frac{\widehat{H}_{n-1}(z/\sqrt{n})}{\widehat{H}_n(z/\sqrt{n})}\leq \frac{1}{z^2},
\quad \hbox{for} \ | z | \leq \frac{1}{3}\,.
\end{align}
On the other hand using Lemma \ref{powers} equation (\ref{deriv_rec}) implies that
\[
 \mathcal{M}_n(z)=n \frac{\widehat{H}_{n-1}(z)}{\widehat{H}_n(z)},
\]
furthermore $\mathcal{M}_n(z)$ is a monotonically increasing (for $z \geq 0$), convex function (due to its definition and
the fact that $M_n(k)=0$ when $k$ is odd).

\noindent Now
$$  \frac{\widehat{H}_{n-1}(x)}{\widehat{H}_n(x)}
=1+\int_0^x{\frac{d}{dy}\frac{\widehat{H}_{n-1}(y)}{\widehat{H}_n(y)}\ dy}
>1+\frac x2\left.\frac{d}{dy}\frac{\widehat{H}_{n-1}(y)}{\widehat{H}_n(y)}\right|_{y=x/2}, $$
since $\frac{\widehat{H}_{n-1}(x)}{\widehat{H}_n(x)}$ is a positive, convex, monotonically increasing function on $\R_{\geq 0}$,
hence $$  \left.\frac{d}{dy}\frac{\widehat{H}_{n-1}(y)}{\widehat{H}_n(y)}\right|_{y=z/\sqrt{n}}\leq \left(\frac{1}{4z^2}-1 \right)\frac{\sqrt{n}}{z}, \quad \hbox{for} \ 0\leq z \leq \frac{1}{3} $$
which means that $$\left.\frac{d}{dy} \frac{\widehat{H}_{n-1}(y)}{\widehat{H}_{n}(y)}\right|_{y=z/\sqrt{n}}=O(\sqrt{n}).$$ Straightforward computation gives that
$$ \frac{d}{dx}\frac{\widehat{H}_{n-1}}{\widehat{H}_n}(x)=\frac n x \left(\frac{\widehat{H}_{n-1}^2(x)}{\widehat{H}^2_n(x)}-\frac{\widehat{H}_{n-2}(x)}{\widehat{H}_n(x)} \right)+\frac 1 x \left(\frac{\widehat{H}_{n-2}(x)}{\widehat{H}_n(x)}-\frac{\widehat{H}_{n-1}(x)}{\widehat{H}_n(x)} \right) $$
hence
 \begin{align*}\left.\frac{d}{dy}\frac{\widehat{H}_{n-1}(y)}{\widehat{H}_n(y)}\right|_{y=z/\sqrt{n}}&=\frac {n\sqrt{n}}{ z} \left(\frac{\widehat{H}_{n-1}^2(z/\sqrt{n})}{\widehat{H}^2_n(z/\sqrt{n})}-\frac{\widehat{H}_{n-2}(z/\sqrt{n})}{\widehat{H}_n(z/\sqrt{n})} \right)+\\&+\frac {\sqrt{n}}{z} \left(\frac{\widehat{H}_{n-2}(z/\sqrt{n})}{\widehat{H}_n(z/\sqrt{n})}-\frac{\widehat{H}_{n-1}(z/\sqrt{n})}{\widehat{H}_n(z/\sqrt{n})} \right)\end{align*}
and so

\begin{equation}\label{change_form}
\frac{\widehat{H}_{n-1}^2(z/\sqrt{n})}{\widehat{H}^2_n(z/\sqrt{n})}-\frac{\widehat{H}_{n-2}(z/\sqrt{n})}{\widehat{H}_n(z/\sqrt{n})}=O\left(\frac{1}{n}\right), \quad \hbox{for} \ 0\leq z \leq \frac{1}{3}\,.
\end{equation}
Now let $f_n(z):=\frac{\widehat{H}_{n-1}(z/\sqrt{n})}{\widehat{H}_{n}(z/\sqrt{n})}$, then equation (\ref{conj_rec}) implies that
$$ 1=f_n(z)-\frac{n-1}{n}z^2f_n(z)f_{n-1}\left( \sqrt{\frac{n-1}{n}}\cdot z\right) $$
and from (\ref{change_form}) it follows that
\begin{align*} f_n^2(z)-f_n(z)f_{n-1}\left( \sqrt{\frac{n-1}{n}}\cdot z\right)&=\frac{\widehat{H}_{n-1}^2(z/\sqrt{n})}{\widehat{H}^2_n(z/\sqrt{n})}-\frac{\widehat{H}_{n-2}(z/\sqrt{n})}{\widehat{H}_n(z/\sqrt{n})}=\\&=O\left(\frac{1}{n}\right). \end{align*}
Therefore if for some fixed $z$ in the interval above $h(z)$ is a limit point of the sequence $f_n(z)$ then it satisfies the
following equation:
\begin{align}\label{fixp_eq} 1=h(z)-z^2h(z)^2 .\end{align}
In fact
\begin{align*}1&=f_n(z)-\frac{n-1}{n}z^2f_n^2(z)
- \frac{n-1}{n}z^2\cdot\left[f_n(z)f_{n-1}\left( \sqrt{\frac{n-1}{n}}\cdot z\right)-f_n^2(z)\right] \\
1&=f_n(z)-\frac{n-1}{n}z^2f_n^2(z)+ O\left( \frac 1 n \right).
\end{align*}

Introducing the notation $c(z) = \sum_{k=0}^\infty C_k z^{2k}$ the usual computation gives that
\begin{eqnarray*}
1- 2z c(z) &=& 1- \sum_{k=0}^\infty \frac{z^{k+1}}{(k+1)!} 2^{k+1} (2k-1)!! \\
&= &
1- \sum_{k=0}^\infty z^{k+1} 2^{2k+2} (-1)^k \left(\begin{array}{c} \frac{1}{2} \\ k+1\end{array}\right) =
\sqrt{1-4z}\,.
\end{eqnarray*}
\\
Thus $c(z)=\frac{1-\sqrt{1-4z}}{2z}$ which is the smaller solution of the equation $1=c(z)-zc(z)^2$, consequently setting 
$h(z)=c(z^2)$ we arrive at (\ref{fixp_eq}) (for more details see \cite{koshy2008catalan} p. 27-28).

Since acording to (\ref{upperbound}) the sequence $f_n(z)$ is uniformly bounded, in order to prove
the convergence of the whole sequence it is enough to prove that
$$ f_n(z)\leq c(z^2) \ \ \ \hbox{for  }  n\geq 1 \hbox{ and  } 0 \leq z \leq \frac{1}{3}\,, $$ 
implying that $c(z^2)$ is the only limit point of this sequence.

 For $n=1$ one has $f_1(z)=1\leq c(z^2)$. Equation (\ref{conj_rec}) implies that
\begin{align*} 1&=\frac{\widehat{H}_{n-1}}{\widehat{H}_n}\left(\frac{z}{\sqrt{n-1}}\right)
-z^2 \frac{\widehat{H}_{n-2}}{\widehat{H}_n}\left(\frac{z}{\sqrt{n-1}}\right)=\\
&= f_n\left(\sqrt{\frac{n}{n-1}} \cdot z\right)-z^2f_n\left(\sqrt{\frac{n}{n-1}}\cdot z\right)f_{n-1}(z).\end{align*}
Let us look at the following map $\xi\mapsto \eta(\xi)$, where
$$ 1=\eta(\xi)-z^2\eta(\xi)\xi, $$
and $z\in[0,1/3]$ is arbitrary, but fixed. Note that the fixed points of $\eta(\xi)$ are the same as the solutions to (\ref{fixp_eq}).
 Expressing $\eta$ in term of $\xi$ we get
\begin{align*}
    \eta(\xi)=\frac{1}{1-z^2\xi}.
\end{align*}
Observe that $\eta$ is a strictly monotonically increasing function on the set $\xi < \frac{1}{z^2}$ with $\eta(0)=1\leq c(z^2)$, and so $\eta(\xi)\leq c(z^2)$ if $\xi\leq c(z^2)$. Now put $\xi=f_{n-1}(z)$, then $\eta=f_{n}(\sqrt{n/(n-1)}\cdot z)\leq c(z^2)$. Since $f_n$ is monotonically increasing we have
$$ f_n(z)\leq f_{n}\left(\sqrt{\frac{n}{n-1}}\cdot z\right)\leq c(z^2)  $$as well.
Thus the only accumulation point of $(f_n(z))_{n\in \N}$ is $c(z^2)$. Remember that $\frac 1 n \mathcal{M}_n(z/\sqrt{n})=f_n(z)$,
hence the proof of (\ref{limit_mn}) is complete.

Now the convergence of the power series on the interval $[0, \frac{1}{3}]$ implies the convergence of the
coefficients, so
\[
\frac{1}{2^kn^{\frac{k}{2}+1}} M_n (k) \rightarrow \begin{cases} \frac{C_{k/2}}{2^k} & \hbox{k is even}\,, \\
0 & \hbox{k is odd}\,.
\end{cases}
\]
Shortly, they tend to the corresponding moments of the semicircle distribution since$$ \int_\R{x^k\rho_{sc}(x)dx}= \begin{cases} \frac{C_{k/2}}{2^k} & \hbox{k is even}\,, \\
0 & \hbox{k is odd}. \end{cases}$$ This concludes the proof of the theorem.

\end{proof}
\section{The sum of the $k^{th}$ power of the roots of Hermite polynomials}

In this section we are going to prove a stronger statement than the one in the previous section, namely:

\begin{theorem}\label{Hermite_moment} 
$M_n(k)$ is a polynomial in $n$, where $M_n(k) = 0$ when $k$ is odd, while
\[
\deg_n M_n(k) = \frac{k}{2} + 1
\]
when $k$ is even. In these cases the coefficient of $n^{k/2 + 1}$ in $M_n(k)$ is given by the Catalan number $C_{k/2}$.
In particular,
    \begin{align} M_n(k)=\begin{cases}
        n^{k/2+1} C_{k/2} + f(n)\,, & \mbox{ k is even; } \\ 0\,, & \mbox{ k is odd, }
    \end{cases} \end{align}
    where $C_{k/2}=\frac{\binom{k}{k/2}}{k/2+1}$ and $f$ is a polynomial of degree at most $k/2$. 
\end{theorem}

Before the proof of theorem \ref{Hermite_moment} we state a well known result without proof:
\begin{proposition}[See \cite{Sze39} p. 106. eqn. 5.5.4.]\label{Hermite_coeff}
Let us denote by $H_n(x) = \sum_{j=0}^n a_j^{(n)}  x^j$ the Hermite polynomial of degree $n$.
Then
\begin{align}\label{coeff_a}
        a_{n-k}^{(n)}=\begin{cases}(-1)^{k/2}\frac{\binom{n}{k}k!}{(k/2)!2^{k/2}} & k \ \mbox{ is even;} \\ 0 & k\ \mbox{ is odd.} \end{cases}
    \end{align}
\end{proposition}

\begin{proof} (of Theorem \ref{Hermite_moment}).
 First let us note that $a^{(n)}_{n-k} = 0$ when $k$ is an odd number or $k>n$, thus by induction we obtain
     that $M_n(k)=0$ for $k= 1, 3, \dots$. Using this fact let us write Newton's identities (\ref{Newton}) in the following matrix form:
    \begin{align}
        \begin{bmatrix}
            1  & 0 & \ldots & \ldots & 0 \\
            a_{n-2}^{(n)} & 1 & 0&  \ldots & 0 \\
             & &  \ddots & & \vdots \\
	& & &\ddots & 0 \\
            a_{n-2(k-1)}^{(n)} &  \ldots &  \ldots & a_{n-2}^{(n)} & 1
        \end{bmatrix} \begin{bmatrix}
            M_n(2) \\ M_n(4) \\ \vdots \\  \vdots \\ M_n(2k)
        \end{bmatrix}=\begin{bmatrix}
            -2a_{n-2}^{(n)} \\ -4 a_{n-4}^{(n)} \\ \vdots  \\ \vdots \\  -2ka_{n-2k}^{(n)}
        \end{bmatrix}.
    \end{align}
Since the determinant of the matrix standing on the left hand side is $1$ we obtain that
    \begin{align}\label{matr}
        M_n(2k)=\det \begin{bmatrix}
             1  & 0 & 0 &\ldots &  -2a_{n-2}^{(n)} \\
            a_{n-2}^{(n)} & 1 & 0& \ldots & -4a_{n-4}^{(n)}  \\
             & &  \ddots & \\
             & & &  \ddots \\
            a_{n-2(k-1)}^{(n)} & a_{n-2(k-2)}^{(n)} & \ldots & a_{n-2}^{(n)} & -2ka_{n-2k}^{(n)}
        \end{bmatrix}.
    \end{align}
In order to compute the determinant above let us  introduce the following function of variable $x$:
\begin{align}
        A(k,l):=\det \begin{bmatrix}
            1 & 0 & 0 & \ldots & (x)_{l+1} \\
            -\frac{x(x-1)}{2} & 1 & 0 & \ldots & -\frac{(x)_{l+3}}{2} \\
            & \ddots& \ddots & & \\
            & & &  1 & \vdots\\
            \frac{(-1)^{k-1}(x)_{2(k-1)}}{2^{k-1}(k-1)!} & \ldots & \ldots & -\frac{x(x-1)}{2} & \frac{(-1)^{k-1}(x)_{l+2(k-1)+1}}{(k-1)!2^{k-1}}
        \end{bmatrix}
    \end{align}
for $k\geq 2, l\geq 1$,
    where $(x)_l=x(x-1)\cdots (x-l+1)$ and $(x)_0 = 1$. \\
Observe that $M_{n} (2k) = A (k, 1)$ with $x=n$. Multiply the first column by $(x)_{l+1}$ and subtract it from the last one, then the $j^{th}$ element  $(j\geq 2)$ of the last column is given as:
\begin{align}
        &\frac{(-1)^{j-1}}{2^{j-1}(j-1)!} (x)_{l+2(j-1)+1} -\frac{(-1)^{j-1}}{2^{j-1}(j-1)!} (x)_{l+1} (x)_{2(j-1)}
= \nonumber \\
        &=\frac{(-1)^{j-1}}{2^{j-1}(j-1)!}(x)_{2(j-1)} [(x-2j+2)_{l+1} - (x)_{l+1}]= \nonumber \\
        &=\frac{(-1)^{j-1}}{2^{j-1}(j-1)!}(x)_{2(j-1)}\left[-2(j-1)\right]\sum_{h=0}^{l}\left[(x-2j+2)_{h} \times (x-h-1)_{l-h}\right]
\nonumber \\
       &= \frac{(-1)^{j-2}}{2^{j-2}(j-2)!}\sum_{h=0}^l \left[(x)_{2(j-1)+h} \times (x-h-1)_{l-h}\right]\,.
  \nonumber
  \end{align}
    due to the fact that
    \begin{align}
        \prod_{j=0}^l{\alpha_j}-\prod_{j=0}^l{\beta_j}=\sum_{h=0}^{l} \left(\prod_{j:j<h} \alpha_j (\alpha_h-\beta_h) \prod_{j:j>h} \beta_j  \right)\,.
\nonumber
    \end{align}
Let us observe that now the first element in the last column is zero, while all the other elements can be written as sums
with $l+1$ elements, where in the $i^{\rm{th}}$ summand every element is multiplied by the same factor $(x-h-1)_{l-h}$.
Introducing the notation $i=h+1$ we obtain that for $k\geq 3$, $l\geq 1$
    \begin{align}\label{recursion_A} 
        A(k,l)=\sum_{i=1}^{l+1} A(k-1,i) (x-i)_{l-i+1}.
    \end{align}
For $k=2$, $l\geq 1$
    \begin{align}
        A(2,l)&=\det \begin{bmatrix}
           1 & x(x-1)\cdots (x-l) \\ -\frac{x(x-1)}{2} & -\frac{1}{2}x(x-1)\cdots (x-l-2)
        \end{bmatrix}  \nonumber\\
&=\frac{x(x-1)\cdots (x-l)}{2}\left(x(x-1)-(x-l-1)(x-l-2) \right) \nonumber \\
&=\frac{x(x-1)\cdots (x-l)}{2}((2l+2)x -(l+1)(l+2)) \nonumber
    \end{align}
Let us observe that
$A(2, l)$ can also be written as
\[
A(2, l) = \sum_{i=1}^{l+1} (x)_{i+1} (x-i)_{l-i+1}\,,
\]
hence with the notations $A(1, l) = (x)_{l+1}$, $A(0, 1) = x$ and $A(0, l ) = 0$ for $l\geq 2$ we might extend the validity of the formula (\ref{recursion_A}) for $k=2$ and $k=1$ as well.
\begin{lemma}\label{degA} 
$ \deg A(k,l) =k+l $, for $k\geq 1, l\geq 1$\,.\end{lemma}
\textit{Proof of the lemma:} 
Trivially $\deg A(1, l) = l+1$ and the highest degree coefficients are positive.  
Suppose the claim above is true for $k-1$ and all $l\geq 1$, then
    \begin{align}
       \deg A(k,l)&=\deg \left(\sum_{i=1}^{l+1}\prod_{j=i}^l(x-j) A(k-1,i)\right)  =  k+l\, ,
    \nonumber
 \end{align}
because by induction the highest degree coefficients of $A(k-1,l)$  and that of the multipliers $(x-i)_{l+1-i}$
for $l=1, \dots , l+1$ are positive.
 This concludes the proof of Lemma \ref{degA}. $\Box$ \\

In particular $\deg_n M_n (2k) = \deg A(k, 1) = k+1$. For example, when $k=2,4,6$ and $x=n$ it is easy to see that
\begin{eqnarray} \label{mn2}M_n(2)&=&-2a_{n-2}= n^2-n=C_1\cdot n^2-n \\
\label{mn4} M_n(4)&=&2n^3-5n^2+3n=C_2\cdot n^3-5n^2+3n \\ 
M_n(6)&=&5n^4-22n^3+32n^2-15n=C_3\cdot n^4-22n^3+32n^2-15n \label{mn6}
 \end{eqnarray}

\begin{remark}Let us point out that from this it follows that $\lim_{n\rightarrow \infty} M_n(2k)/n^{k+1}$ equals to
the leading coefficient of $A(k, 1)$, consequently Theorem \ref{Hermite_limit} implies that this has to be $C_k$.
But to provide 
a self-contained proof we show that it is possible to determine this leading coefficient using a simple graph theoretic argument. 
\end{remark}
Thus, we are going to prove that the leading coefficient -- that is the coefficient of $x^{k+1}$ -- in $A(k, 1)$ is 
$C_k=\binom{2k}{k}/(k+1)$. 

Since in the recursive formula for $A(k, l)$ the factors for $A(k-1, i)$ are with leading coefficient one, the leading
coefficient of $A(k, l)$ can be obtained as the sum of that of $A (k-1, 1), \dots , A(k-1, l+1)$. Applying now the recursive
formula for the elements $A(k-1, i)$ and so on, we obtain that the leading coefficient of $A(k,1)$ is given by the number
of $A(1,t)$ terms in the representation of $A(k,1)$ by these elements.
This question can be translated in the following graph theoretical question:\\
    Let us consider the following (directed) graph  $G:=((\Z_{\geq 0})^2,E)$: For $a:=(i_1,j_1),b:=(i_1,j_2)\in (\Z_{\geq 0})^2$ there is an edge from $a$ to $b$, i.e. $(a,b)\in E$ if and only if $i_2=i_1+1$ and $j_2=j_1-1+h$ for $h\geq 0$. Let $a(j)$ denote $a$'s $j^{th}$ coordinate for $j=1,2$. The number of simple (directed) paths from the origin to $(k,0)$ is exactly the coefficient of $x^{k+1}$ in $A(k,1)=M_n(2k)$.
It can be checked easily that for $k=1$ it is $1$, for $k=2$ it is 2, for $k=3$ it is 5. \\
    \begin{lemma} \label{paths} In the graph G the number of simple paths from the origin to $(k,0)$ is exactly $C_k$. \end{lemma}
    \textit{Proof of Lemma \ref{paths}:} By induction on $k$. Denote by $d_{k+1}$ the number of simple paths from the origin to $(k+1,0)$, define $d_0:=1$, denote by $o$ the origin and denote by $\mathcal{P}_{k+1}$ the collection of (directed) simple paths from $(0,0)$ to $(k+1, 0)$, i.e.
    \begin{align*} \mathcal{P}_{k+1}:=\{ & (o,a_1,a_2,\ldots,a_{k+1}) | a_i\in (\Z_{\geq0})^2,\\ & \ a_i(1)=i,\ (a_{i-1},a_{i})\in E,\ 1\leq i \leq k+1, a_{k+1}(2)=0 \ \}. ,\end{align*}
     For any path $P=(o,\ldots,a_{m})$ set
     $$ t(P):=\inf\{j\geq 1| a_j=(j,0) \} $$
     and let $\mathcal{P}_{k+1}(i):=\{P\in\mathcal{P}_{k+1} | t(P)=i+1 \}$. Note that $t(P)=i+1$ means that the first node of the path $P$ whose second coordinate is 0 and differs from the origin is $a_{i+1}$.
     It is easy to see that $\mathcal{P}_{k+1}=\bigcup\limits_{0\leq i \leq k}^*\mathcal{P}_{k+1}(i),$ hence $|\mathcal{P}_{k+1}|=\sum_{i=0}^k|\mathcal{P}_{k+1}(i)|.$
     We want to show that
\[
|\mathcal{P}_{k+1}(i)|=|\mathcal{P}_{i}||\mathcal{P}_{k-i}|=d_{i}d_{k-i}\,.
\]
 Note that $\mathcal{P}_0=\{(o)\}$ and $\mathcal{P}_1=\{\left(o,(1,0)\right)\}$. Given two paths $P_1=(o,a_1,\ldots,a_{i})$ $\in \mathcal{P}_{i}$ and $P_{2}=(o,b_1,\ldots,b_{k-i})\in\mathcal{P}_{k-i} $ one can make a path $P\in \mathcal{P}_{k+1}(i)$ as follows: let us construct $P=(o,c_1,\ldots,c_{k+1})$ in such a way that $c_j(1):=a_j(1)$, $c_j(2):=a_j(2)+1$ for $1\leq j \leq i$ and $c_{j}(1):=b_{j-i-1}(1)+i+1$ , $c_j(2):=b_{j-i-1}(2)$ for $i <  j\leq k+1 $. Remember that $b_0=o$. It is trivial that $c_{k+1}=(k+1,0)$ and due to the definition of the graph $(c_j,c_{j+1})\in E$ for $0\leq j\leq k+1$. Since $a_j(2)\geq 0$ for $1\leq j \leq i$, so $c_j(2)\geq 1$ for these indices, while
$c_{i+1}=(i+1,0)$ thus we obtain that $t(P)=i+1$, therefore $P\in \mathcal{P}_{k+1}(i)$.
     Now take a path $P=(o,c_1,\ldots,c_{k+1})\in \mathcal{P}_{k+1}(i)$. Let $P_1=(o,a_1,\ldots,a_i)$ be defined by $a_j(1):=c_j(1)$ and $a_j(2):=c_j(2)-1$ for $1\leq j\leq i$. Since $t(P)=i+1$, one has  (for $i\geq 1$) that $c_j\geq 1$ for $1\leq j\leq i$, therefore $a_j(2)\geq 0$ for $1\leq j \leq i$. Furthermore, for $i\geq 1$, $c_{i}=(i,1)$ -- because of $c_{i+1}=(i+1,0)$ -- implying that $a_i=(i,0)$. Obviously $P_1$ is a valid path (due to the structure of the graph) and its last node is $(i,0)$.
     The number of such paths is given by $d_i$. Define $P_2=(o,b_1,\ldots,b_{k-i})$ by $b_j(1)=c_{j+i+1}(1)-i-1$ and $b_j(2)=c_{j+i+1}(2)$. Obviously $b_0=o=(0,0)$, and $b_{k-i}(1)=c_{k+1}(1)-i-1=k-i$, $b_{k-i}(2)=c_{k+1}(2)=0$, hence $P_2\in \mathcal{P}_{k-i}$. The number of such paths is $d_{k-i}$. (Note that in case of $i=0$,  $\mathcal{P}_0=\{(o)\}$, $P_1=(o)$ and $P_2=(o,b_1,\ldots,b_k)$, so one only has to concatenate the two paths in such a way that the coordinates of the nodes of the path are valid, i.e. $c_0=o$, $c_1=(1,0)$, $c_{j+1}=(b_j(1)+1,b_j(2))$ for $1\leq j\leq k$ . If $P\in\mathcal{P}_{k+1}(0)$ is given, then $P=(o,(1,0),c_2,\ldots,c_{k+1})$ thus $P_1:=(o)$ and $P_2=(o,b_1,\ldots,b_k)$ with $b_j(1)=c_{j+1}(1)-1$ and $b_j(2)=c_{j+1}(2)$ for $1\leq j\leq k$.)
     Now it is easy to see that we found a bijection between $\mathcal{P}_{k+1}(i)$ and $\mathcal{P}_{i}\times \mathcal{P}_{k-i}$, hence  
    \begin{align}
        d_{k+1}=\sum_{i=0}^k d_id_{k-i}\,.
    \end{align}
Thus the sequence $d_0, d_1, \dots $ satisfies the same recursion which is valid for the Catalan numbers. Since as we pointed out above
the first two terms of these sequences coincide we obtain by induction that
$d_k= C_k = \binom{2k}{k}/(k+1)$ for $k\geq 0$, thus Lemma \ref{paths} is proved.$\Box$ \\ 
Summarizing what we know until this point we arrive at
$$ A(k,1)=C_k x^{k+1}+ f(x) $$
where $f$ is a polynomial of degree $k$. This concludes the proof our Theorem \ref{Hermite_moment}.
\end{proof}

We would like to point out that this methodology enables us to also compute
the coefficient of the second highest degree term as well.
\begin{proposition}\label{SecondC}
	Let $A(k,1)=C_k x^{k+1}+s_kx^{k} + g(x)$, where $g$ is a polynomial of degree $k-1$ at most. Then
\begin{equation}\label{s_k} s_k=-\left( 2^{2k-1}-\binom{2k-1}{k} \right). \end{equation}
\end{proposition}
\begin{proof}
 First observe that $s_0 = 0$, due to the identity $A(0, 1) = x$. Next we are going to show that the following recursion holds:
\begin{equation} \label{SecCoeff} 
s_k= \sum_{j=1}^k\left(s_{k-j} C_{j-1}+C_{k-j}(s_{j-1}-jC_{j-1})\right)\ \  \mbox{   for } k\geq 1. 
\end{equation}
Using the notations of Lemma \ref{paths} we have
$$ \mathcal{P}_k=\bigcup_{j=0}^{k-1}{\mathcal{P}_k(j)}  $$
We showed in Lemma \ref{paths} that $|\mathcal{P}_k|$ is equal to the highest degree coefficient of $A(k,1)$ that is 
$C_k$.  Let us write on any edge $(a, b)$ of the graph the following polynomials:
\begin{equation} \label{poly} 
(a,b)\mapsto \begin{cases} 1 & \mbox{ if } b(2)=a(2)-1 \\ \prod_{j=a(2)+1}^{b(2)+1}(x-j) & \mbox{ if } b(2)\geq a(2) \end{cases} 
\end{equation}
and assign the polynomial $x$ to the origin. Using this we can assign polynomials to each path $P$ in the graph starting at the origin
as the product of the polynomials assigned to the edges along the path -- denote this by $p(x; P)$ -- and the one written on the origin.
We obtain $xp(x; P)$.
Recursion (\ref{recursion_A}) implies that $A(k,l)$ equals the sum of the polynomials corresponding to paths leading from the origin to $(k,l-1)$. 
Especially, we have that 
\begin{align*} A(k,1)=x\sum_{P\in \mathcal{P}_k}p(x; P)=
x \sum_{j=1}^{k}{\sum_{ \substack{P\in \mathcal{P}_k(j-1)}}p(x; P) }.
\end{align*}
Let us observe, that the second highest degree coefficients are always negative. Furthermore, for any $P\in \mathcal{P}_k(j-1)$ one has
that $p(x; P)=q_1(x)q_2(x)$, where the polynomial $q_1(x)$ corresponds to a path starting from the origin, ending in $(j,0)$ and never touching the x-axis before that, while the polynomial $q_2(x)$ corresponds to the path from $(j,0)$ to $(k,0)$.
 Due to translation invariance of the graph the polynomial $q_2(x)$ coincides with a polynomial corresponding 
to a path from the origin to $(k-j,0)$. 
 Hence
\begin{eqnarray}
 A(k,1)&=&x\sum_{j=1}^k\sum_{\substack{ P\in \mathcal{P}_j(j-1)\\ Q\in \mathcal{P}_{k-j}}} p(x; P) p(x; Q)
= x \sum_{j=1}^k\left(\sum_{ P\in \mathcal{P}_j(j-1)} p(x; P)
\sum_{Q\in \mathcal{P}_{k-j}} p(x; Q)\right)\nonumber\\
&=& \sum_{j=1}^k\left(\sum_{ P\in \mathcal{P}_j(j-1)} p(x; P)
A(k-j, 1)\right)\label{split_recursion}
\,.
\end{eqnarray} 

Let us recall that in the proof of Lemma \ref{paths} we have constructed a bijection between 
$\mathcal{P}_{j-1}$ and $\mathcal{P}_j(j-1)$. Roughly speaking starting with a path in $\mathcal{P}_{j-1}$ keeping
the origin as a starting point but increasing the second coordinates of the other points along the path by one and finally adding
a last edge from $(j-1, 1)$ to $(j, 0)$ we obtained the corresponding trajectory. The map (\ref{poly}) shows that as a result of this construction all the roots of the polynomial corresponding to the path in  $\mathcal{P}_{j-1}$ will be increased by $1$. Since its degree
is $j$ the sum of the roots increases by $j$. Taking the summation with respect to all paths in $\mathcal{P}_{j-1}$ we obtain
that the highest degree coefficients of $\sum_{ P\in \mathcal{P}_j(j-1)} p(x; P)$ and $A(j-1, 1)$ are equal, while
the difference in the second highest degree coefficient is $j C_{j-1}$, thus equation (\ref{SecCoeff}) holds.

This recursion leads to the generating function 
\begin{align*} \mathcal{S}(z)&=\sum_{k=1}^\infty{s_kz^k}=-\frac{zc(z)\frac{d}{dz}(zc(z))}{1-2zc(z)}=-\frac{z}{1-4z}c(z)=\\
&= -\sum_{k=1}^\infty \sum_{j=1}^{k}C_{k-j}4^{j-1} z^k .
 \end{align*}
In order to determine this value more explicitely let us consider the symmetric walk on $\Z$ with $2k-1$ steps.  
The number of all possible trajectories is $2^{2k-1}$. 
Write the set of possible trajectories as the disjoint union of paths that enter the negative region in the $(2j+1)^{th}$ step first with $0\leq j \leq k-1$ and those that never enter the negative region. Rewriting (\ref{s_k}) in the following way
\begin{equation} \label{otherpaths} \sum_{j=0}^{k-1} {C_j2^{2(k-j-1)}} \end{equation} 
one has that this sum \label{otherpaths} counts the trajectories of the former type, while the number of trajectories of the latter type is given by $\binom{2k-1}{k}$ (see e.g. \cite{bil13} p.71), hence
$$ s_k=- \left(2^{2k-1}-\binom{2k-1}{k} \right) $$ 
and so the proof is complete. 
\end{proof}

\begin{remark}
 Note that Proposition \ref{SecondC} implies that the convergence rate in Theorem \ref{Hermite_SC} cannot be faster than $O(1/n)$.
\end{remark}

\section{Concluding remarks}
\label{Conc}

In the introduction we have seen, that even under general conditions on the random symmetric matrix the expectation of its characteristic polynomial is the monic Hermite polynomial of appropiate degree. The limiting distribution of the roots of the $H_n$ is the semicircle law as it is shown in Theorem \ref{Hermite_limit}. It is also known that the limiting distribution of the eigenvalues of a properly scaled Wigner matrix is given by the semicircle law \cite{Ar67} in the same sense as above, hence there is a deep connection between the Hermite polynomials, random symmetric matrices with independent elements and the semicircle law. This suggests that studying the Hermite polynomials and their roots could give us a deeper insight on the behavior of the eigenvalues of a Wigner random matrix.

\bibliography{KM}
\bibliographystyle{plain}

\end{document}